\documentclass[11pt,reqno]{amsart}
\usepackage{amsmath,amssymb,amsthm,amsfonts,enumerate,hyperref}
\usepackage{graphicx}

\theoremstyle{plain}
\newtheorem{theorem}{Theorem}[section]
\newtheorem{proposition}[theorem]{Proposition}
\newtheorem{cor}[theorem]{Corollary}
\newtheorem{lemma}[theorem]{Lemma}

\theoremstyle{definition}
\newtheorem{definition}[theorem]{Definition}

\theoremstyle{remark}

\numberwithin{equation}{section}

\newcommand{\Sym}{\mathrm{Sym}^{2}}

\newcommand{\Riem}{\mathrm{Rm}}

\newcommand{\Ric}{\mathrm{Ric}}

\renewcommand{\div}{\mathrm{div}}

\begin{document}
\title{On the spectrum of the Page and the Chen-LeBrun-Weber metrics}
\begin{abstract}
We give bounds on the first non-zero eigenvalue of the scalar Laplacian for both the Page and the Chen-LeBrun-Weber Einstein metrics. One notable feature is that these bounds are obtained without explicit knowledge of the metrics or numerical approximation to them. Our method also allows the estimation of the invariant part of the spectrum for both metrics. We go on to discuss an application of these bounds to the linear stability of the metrics. We also give numerical evidence to suggest that the bounds for both metrics are extremely close to the actual eigenvalue. 
\end{abstract}
\author{Stuart J. Hall}

\address{Department of Applied Computing, University of Buckingham, Hunter St., Buckingham, MK18 1G, U.K.} 
\email{stuart.hall@buckingham.ac.uk}

\author{Thomas Murphy }
\address{D\'{e}partment de Math\'{e}matique,
Universit\'{e} Libre de Bruxelles,
 \ Boulevard du Triomphe,
B-1050 Bruxelles,
Belgique.}
\curraddr{Department of Mathematics, McMaster University, 1280 Main St. W., Hamilton ON, Canada.}
\email{tmurphy@math.mcmaster.ca}

\maketitle
\section{Introduction}
\subsection{Main results}
The purpose of this paper is to provide some estimates for the first non-zero eigenvalue of the scalar Laplacian of two distinguished Einstein metrics.  The metrics we are interested in are the Page metric \cite{Pa} on $\mathbb{CP}^{2}\sharp\overline{\mathbb{CP}}^{2}$ 
and the Chen-LeBrun-Weber metric \cite{CLW} on $\mathbb{CP}^{2}\sharp2\overline{\mathbb{CP}}^{2}$. The main result we prove is:
\begin{theorem}\label{main}
Let $g_{P}$ denote the Page metric on $\mathbb{CP}\sharp\overline{\mathbb{CP}}^{2}$ with
$$\Ric(g_{P}) = \Lambda g_{P}, \ \Lambda>0 .$$
Then the first non-zero eigenvalue of the Laplacian on functions $\lambda^{P}_{1}$ satisfies
$$\frac{4}{3}\Lambda<\lambda^{P}_{1} \leq 1.89\Lambda.$$
Let $g_{CLW}$ denote the Chen-LeBrun-Weber metric on $\mathbb{CP}\sharp2\overline{\mathbb{CP}}^{2}$ with
$$\Ric(g_{CLW}) = \Lambda g_{CLW}, \ \Lambda>0 .$$
Then the first non-zero eigenvalue of the Laplacian on functions $\lambda^{CLW}_{1}$ satisfies
$$\frac{4}{3}\Lambda<\lambda^{CLW}_{1} \leq 2.11\Lambda.$$   
\end{theorem}  
The lower bound of $4\Lambda/3$ in Theorem \ref{main} is just the classical Lichnerowicz-Obata lower bound \cite{Bes}. The main contribution of this paper is the upper bound for the first eigenvalue. Motivations for this sort of result come from at least two sources. Firstly, $\lambda_{1}$ is an important quantity to many physicists. For example, it controls  the rate of convergence of heat flow on the manifold. One of the main applications of numerical approximations to Einstein metrics has been to calculate $\lambda_{1}$ \cite{Braun, Donum, DoranC}.\\ 

Secondly, such bounds are useful in the study of the Ricci flow and can be used to determine whether an Einstein metric is linearly stable as a fixed point of the flow (we refer the reader to section \ref{LinStab} for details). The investigation of linear stability was instigated by Cao, Hamilton and Ilmanen \cite{CHI}. They noted that if  the first non-zero eigenvalue $\lambda_{1}$ of the scalar Laplacian satisfies
\begin{equation}\label{CHIBound}
\lambda_{1} < 2 \Lambda
\end{equation}
then the Einstein metric $g$ is linearly unstable and can be destabilised by conformal perturbations. They raised as an open question the existence of any Einstein metric satisfying the bound (\ref{CHIBound}).  Theorem \ref{main} answers this question in the affirmative and gives the following corollary:
\begin{cor}
The Page metric is linearly unstable and can be destabilised by conformal perturbations.
\end{cor}
The instability of the Page metric has been known for nearly thirty years due to the work of Young \cite{You} (though it seems that the mathematical community was not aware of her work until recently).  In the recent paper \cite{HHS} the first author, Robert Haslhofer and Michael Siepmann gave an alternative proof of the instability of the Page metric based on the presence of many ($>1$) harmonic 2-forms on this manifold.  There the Bunch-Donaldson numerical approximation to the Chen-LeBrun-Weber metric \cite{BunDon} was used to give strong evidence that the Chen-LeBrun-Weber metric is also unstable.\\

Our methods do not need any numerical approximations to the metrics but unfortunately the bound $2.11$ is tantalisingly just above the magic number $2$ that is needed to show instability.  In section \ref{Disc} we give some heuristic reasoning as to why one might expect this bound to be very close to optimal.  This is reinforced by the numerics in section \ref{NumRes} which suggest that $1.89\Lambda$ is very close to the exact value of $\lambda_{1}^{P}$ and $2.11\Lambda$ is close to $\lambda_{1}^{CLW}$ (assuming that the first non-zero eigenvalue lies in the invariant part of the spectrum). 
\subsection{Structure of the paper}
The method of proving Theorem \ref{main} is extremely simple. We use the characterisation of the first non-zero eigenvalue $\lambda_{1}$ given by the Rayleigh quotient
$$\lambda_{1} = \inf \left\{\frac{\|\nabla f\|_{L^{2}}^{2}}{\|f\|^{2}_{L^{2}}} : f \in C_{0}(M)\right\},$$
 where $C_{0}(M)$ is the space of all functions with integral $0$. Hence evaluating the quotient on any test function (normalised to have integral 0) gives an upper bound for the eigenvalue. Of course the problem with doing this, especially  for the CLW metric, is that one needs to know the metric as well as the volume form in order to evaluate the term $\|\nabla f\|_{L^{2}}$ .\\ 

In section \ref{RalQuo} we explain how the Page and CLW metrics are conformal to a K\"ahler metric. We show how this can be used to simplify the calculation of the Rayleigh quotient. In section \ref{TorKal} we use the fact that the K\"ahler metrics are toric-K\"ahler to explicitly evaluate the integrals given in section \ref{RalQuo}, thus proving the main theorem. In section \ref{Disc} we examine the bounds and explain how they relate to the classical Lichnerowicz-Matsushima bound. In section \ref{LinStab} we give more details on the relationship between the spectrum of the scalar Laplacian and the notion of linear stability. Finally, in section \ref{NumRes} we investigate the spectrum using a more general Rayleigh-Ritz method.  This involves finding a suitable set of test functions $$T_{N} = \{\psi_{1}, \psi_{2}, ...,\psi_{N}\}$$ and then computing the $N \times N$ matrices
$$A_{ij} = \langle \nabla \psi_{i} ,\nabla \psi_{j} \rangle_{L^{2}} \textrm{ and } B = \langle \psi_{i} ,\psi_{j} \rangle_{L^{2}}.$$ 
One then hopes that the eigenvalues of $B^{-1}A$ will converge to the eigenvalues of $\Delta$.  As we pick very symmetric test functions, we may only be able to compute the symmetric part of the spectrum which can be strictly smaller than the whole spectrum \cite{AbrFr}. Where there is convergence, one is able to find the corresponding eigenfunctions expanded in terms of the test functions.
\subsection{Notation and conventions}
We will use the convention that the Laplacian has non-negative eigenvalues. We will show that the calculation of the Rayleigh quotients we use could be written as a functions of a single variable $a$ (which determines the critical K\"ahler class in each case). The value of $a$ can be approximated to any order as it is the root of a polynomial. Where appropriate, will we give values to 4 significant figures.\\
\\
\emph{Acknowledgements:}
SH would like to thank his doctoral advisor Simon Donaldson for introducing him to many of the ideas we have used in this paper. We would like to thank Robert Haslhofer for his interest and comments on a previous version of this paper and the anonymous referees for numerous suggestions for improvements. We would also like to thank Steve Zelditch for his assistance.
TM is supported by an A.R.C. grant.
We acknowledge the support of a Dennison research grant from the University of Buckingham which funded a research visit by TM.

\section{Simplifying the Rayleigh quotient}\label{RalQuo}
The purpose of this section is to exploit some basic facts about conformally K\"ahler, Einstein $4$-manifolds in order to reduce the calculation of the Rayleigh quotient to integrals involving only rational functions of the scalar curvature of the K\"ahler manifold.\\
\\
The Page metric on $\mathbb{CP}^{2}\sharp\overline{\mathbb{CP}}^{2}$ has a cohomogeneity one action by $U(2)$ which reduces the Einstein equation to a non-linear system of ODEs which can be solved explicitly. Hence, in theory, one could compute arbitrarily many eigenvalues using a Rayleigh-Ritz method (see section \ref{NumRes} for this approach). Unfortunately the CLW metric on $\mathbb{CP}^{2}\sharp2\overline{\mathbb{CP}}^{2}$ only admits a cohomogeneity two action by a torus $\mathbb{T}^{2}$ and so the Einstein equation is given by a non-linear system of PDEs.  The existence proof given by the authors in \cite{CLW} is non-constructive, making obtaining information about the geometry of the metric extremely difficult.  The main reason we can make progress is a wonderful feature both metrics share. This is a link with K\"ahler geometry that was first noticed by Derdzinski \cite{Derd}.  We recall that an extremal K\"ahler metric is one where the gradient of the scalar curvature is a real holomorphic vector 
field.

\begin{proposition}[Derdzinski]\label{Derd}
Let $(M^{4},h)$ be a connected oriented Einstein manifold such that $W^{+}$ has at most 2 distinct eigenvalues at each point. Then either $W^{+}\equiv 0$, or else $W^{+}$ has exactly 2 eigenvalues at each point. In the latter case, moreover, the conformally related metric $g={(24)}^{1/3}|W^{+}|^{2/3}h$ is locally conformally K\"ahler. The scalar curvature $s$ of $g$ is then nowhere zero and $h=s^{-2}g$. Furthermore, the metric $g$ is an extremal K\"ahler metric. 
\end{proposition}
LeBrun used this observation to prove the following structural result for non-K\"ahler, Hermitian Einstein metrics on complex surfaces. 
\begin{proposition}[LeBrun]\label{LeBrun}
Let $(M^{4},J,g_{e})$ be a compact non-K\"ahler, Einstein Hermitian manifold then there is an extremal K\"ahler metric $g_{k}$ on $(M,J)$ with non-constant scalar curvature $s_{k}$ such that $g_{e}=s_{k}^{-2}g_{k}$.
\end{proposition} 
Both the Page and the CLW metrics are Hermitian and so are conformal to extremal K\"ahler metrics by LeBrun's result. The following proposition is the technical heart of this paper. It shows that one can compute $\langle \nabla_{e}s_{k}^{p},\nabla_{e}s_{k}^{q}\rangle_{L^{2}(g_{e})}$ as integrals involving only rational functions of $s_{k}$ and the K\"ahler metric $g_{k}$. As we shall see in section 3, this enables explicit calculations in both the case of the Page metric and the CLW metric.
  
\begin{proposition}\label{Sprop}
Let $(M^{4},g_{k})$ be a Riemannian manifold and let $s_{k}$ be the scalar curvature of $g_{k}$. Let $\kappa$ be the scalar curvature of the metric $g_{e} = s^{-2}_{k}g_{k}$.  Then for $p+q \neq 1$ we have the following formula

\begin{equation}\label{ests}
\int_{M}\langle \nabla_{e}s^{p}_{k}, \nabla_{e}s^{q}_{k}\rangle dV_{e}= \frac{pq}{6(p+q-1)}\int_{M}(s_{k}^{4}-\kappa s_{k})s^{p+q-5}dV_{k}.
\end{equation}

In particular we have

\begin{equation}\label{estsrecip}
\int_{M}|\nabla_{e} s_{k}^{p}|^{2}dV_{e}= \frac{p^{2}}{6(2p-1)}\int_{M}(s_{k}^{4}-\kappa s_{k})s^{2p-5}dV_{k},
\end{equation}
where $dV_{e}$ and $dV_{k}$  are the volume forms for $g_{e}$ and $g_{k}$ respectively. 
\end{proposition}
\begin{proof}
We begin by noting the formula for how the scalar curvature of a $4$-manifold changes under conformal rescaling cf \cite{Bes}. If $g_{1} = \phi^{2}g_{e}$ then 
\begin{equation}\label{confs}
s_{1}\phi^{3} = 6\Delta_{e}\phi+\phi s_{e}
\end{equation}
where $s_1$ and $s_{e}$ are the scalar curvatures of $g_{1}$ and $g_{e}$ respectively.
The proof follows from noting that
$$\langle \nabla_{e}s_{k}^{p},\nabla_{e}s_{k}^{q}\rangle = pqs_{k}^{p+q-2}|\nabla_{e}s_{k}|^{2}$$
and
$$s_{k}^{p}\Delta_{e} s_{k}^{q} = qs_{k}^{p+q-1}\Delta_{e} s_{k}-q(q-1)s^{p+q-2}|\nabla_{e}s_{k}|^{2}.$$
Hence
$$q(q-1)s_{k}^{p+q-2}|\nabla_{e}s_{k}|^{2}+s^{p}_{k}\Delta_{e} s^{q}_{k}=qs_{k}^{p+q-1}\frac{1}{6}(s_{k}^{4}-\kappa s_{k})$$
and so
$$\frac{(q-1)}{p}\langle \nabla_{e}s_{k}^{p},\nabla_{e}s_{k}^{q}\rangle+s^{p}_{k}\Delta_{e} s^{q}_{k}=qs_{k}^{p+q-1}\frac{1}{6}(s_{k}^{4}-\kappa s_{k}).$$
The result follows from integrating by parts and noting that $dV_{e} =s^{-4}dV_{k}$.
\end{proof}
In order to use the above proposition, we need to be able to calculate the scalar curvature $\kappa$ of the Einstein metric in terms of data involving the K\"ahler metric $g_{k}$. This is achieved by the following 
\begin{lemma}
Let $(M^{4},g_{e})$ be an Einstein metric satisfying $\Ric(g_{e})=\Lambda g_{e}$. Suppose further that $g_{e}=s^{-2}_{k}g_{k}$ for a K\"ahler metric  $g_{k}$ with scalar curvature $s_{k}$. Then
\begin{equation}\label{EinConst}
\Lambda = \sqrt{\frac{96\pi^{2}\chi(M)+144\pi^{2}\tau(M)-\int_{M}s^{2}_{k}dV_{g_{k}}}{8Vol(g_{e})}},
\end{equation}
where $Vol(g_{e})=\int_{M}s^{-4}_{k}dV_{k}$ is the volume of $M$ with respect to the Einstein metric $g_{e}$.
\end{lemma}
\begin{proof}
We begin by recalling the Allendoerfer-Weil version of the Gauss-Bonnet theorem for Einstein metrics in dimension 4; 
$$\chi(M) = \frac{1}{8\pi^{2}}\int_{M}|W(g_{e})|^{2}+\frac{2\Lambda^{2}}{3}dV_{e},$$
where $\chi(M)$ is the Euler characteristic of $M$ and $W(g_{e})$ is the Weyl curvature of $g_{e}$. The term
$$\int_{M}|W(g_{e})|^{2}dV_{e}$$
is conformally invariant and so we can compute it with respect to the K\"ahler metric $g_{k}$. We also recall the Hirzebruch signature formula (valid for any metric $g$)
$$\tau(M) = \frac{1}{12\pi^{2}}\int_{M}|W^{+}(g)|^{2}-|W^{-}(g)|^{2}dV_{g}$$
where $\tau(M)$ is the signature of $M$ and $W^{+}(g), W^{-}(g)$ are the self-dual and anti self-dual components of the Weyl curvature of $g$ respectively.  Putting all this togther with the pointwise equality
$$|W(g)|^{2}=|W^{+}(g)|^{2}+|W^{-}(g)|^{2},$$
we arrive at
$$\Lambda^{2}Vol(g_{e}) = 12\pi^{2}\chi(M) +18\pi^{2}\tau(M)-3\int_{M}|W^{+}(g_{k})|^{2}dV_{k}.$$
In order to evaluate the last integral we use a standard fact from K\"ahler geometry that
$$|W^{+}(g_{k})|^{2} = \frac{s_{k}^{2}}{24}.$$
The formula for $\Lambda$ now follows.
\end{proof}

\section{Toric-K\"ahler metrics}\label{TorKal}
As mentioned in the previous section, the K\"ahler metrics that are conformal to both the Page and CLW metrics happen to belong to a special class of metric called \emph{extremal toric-K\"ahler} metrics.  There is a rich and deep theory that these metrics fit in to and we refer the reader to Simon Donaldson's survey of the area \cite{Dontor} for background.\\
\\
The essential features of the $4$-dimensional theory that we will use are the following:
\begin{itemize}
\item There is an open set $M^{\circ}\subset M$ with $M^{\circ} \cong P^{\circ} \times \mathbb{T}^{2}$
\item $P \subset \mathbb{R}^{2}$ is a convex polytope known as the \emph{moment polytope}
\item The volume form in the $P \times \mathbb{T}^{2}$ coordinates  is 
$$dV_{k}=dx_{1}\wedge dx_{2} \wedge d\theta_{1} \wedge d\theta_{2}$$
\item The scalar curvature is an affine function of the coordinates on the moment polytope, i.e.
$$s_{k} = c_{1}x_{1}+c_{2}x_{2}+c_{3}$$
for constants $c_{i}$. In fact both metrics are symmetric under an additional $\mathbb{Z}_{2}$ action $x_{1} \leftrightarrow x_{2}$ and so $c_{1}=c_{2}$.
\end{itemize}
We note that the convention we follow in this paper is that the torus fibres have volume $4\pi^{2}$.  This is different to the convention followed in \cite{Dontor}. Putting all these facts together it is not hard to see that the integral of any function of the scalar curvature (especially any rational function) would be easy to compute explicitly as one would be integrating a function in two variables over a polytope in $\mathbb{R}^{2}$.\\  

The moment polytope $P$ is essentially determined by the K\"ahler class
\mbox{$[\omega_{k}] \in H^{2}(M,\mathbb{R})$}. The K\"ahler classes that contain the extremal metrics $g_{k}$ are themselves very special.  They are the classes that contain extremal metrics with the least Calabi energy.  We will not discuss this further but this fact enables the K\"ahler classes to be determined explicitly. We will now give the proof of the main theorem. 

\begin{proof} [Proof of Theorem \ref{main} for the Page metric]
Here we follow the description of the metric given in \cite{HHS}.  This description is originally due to Abreu \cite{Abr} and the existence of the extremal metric is due to Calabi \cite{Cal}. The fact that the metric is actually $U(2)$-invariant allows a concrete description of the metric in this case.\\
\\
The moment polytope is a trapezium (trapezoid) $T \subset \mathbb{R}^{2}$ given as the set of points $(x_{1},x_{2}) \in \mathbb{R}^{2}$ satisfying  the inequalities $l_{i}(x)>0$ where
$$ l_{1}(x) = x_{1}, \ l_{2}(x) = x_{2}, \ l_{3}(x) = (1-x_1-x_2) , \  l_{4}(x) = (x_1+x_2-a).$$  
Here $a$ is a constant $0<a<1$ that determines the K\"ahler class by varying the volume of the exceptional divisor.  As mentioned in \cite{HHS} the class containing the K\"ahler metric conformal to the Page metric is the only root of
$$ 1-6a^{2}-16a^{3}+9a^{4}=0 $$
in the interval $(0,1)$.  Even though it can be explicitly described, we will take  $a \approx 0.3141$ to 4 significant figures.\\
\\
The scalar curvature of the extremal metric is given by 
$$s_{k}(x_{1},x_{2}) = c_{1}(x_{1}+x_{2})+c_{2},$$ 
where
$$c_{1} = \frac{48a}{(1-a)(1+4a+a^{2})} \textrm{ and } c_2 = \frac{12(1-3a^{2})}{(1-a)(1+4a+a^{2})}.$$
The following explicit formulae for integrals of powers of $s_{k}$ make it very clear that we can obtain as high precision as required by computing more of the decimal expansion of $a$. We first note that the integral over the trapezium can be simplified as
$$\int_{T} \left(c_{1}(x_{1}+x_{2})+c_{2}\right)^{q}dx_{1}dx_{2} = \int_{a}^{1}(c_{1}t+c_{2})^{q}t dt.$$
Thus when $q \neq-1,-2$ we have:
$$\int_{T} \left(c_{1}(x_{1}+x_{2})+c_{2}\right)^{q}dx_{1}dx_{2} = \left[\frac{(c_{1}t+c_{2})^{q+1}}{(q+1)c_{1}}\left(t-\frac{(c_{1}t+c_{2})}{(q+2)c_{1}}\right)\right]^{1}_{a}.$$
When $q=-1$ we have
$$\int_{T}\left(c_{1}(x_{1}+x_{2})+c_{2}\right)^{-1}dx_{1}dx_{2} = -\left[\frac{c_{2}}{c_{1}^{2}}\log(c_{1}t+c_{2})-c_{1}t\right]_{a}^{1},$$
and when $q=-2$ the formula is
$$\int_{T}\left(c_{1}(x_{1}+x_{2})+c_{2}\right)^{-2}dx_{1}dx_{2} =\left[-\frac{t}{c_{1}(c_{1}t+c_{2})}\right]_{a}^{1}+\left[\frac{1}{c_{1}^{2}}\log(c_{1}t+c_{2})\right]_{a}^{1}.$$
The volume of the Page metric in this representation is 
$$Vol(g_{P}) = \int_{M}s_{k}^{-4}dV_{k} =4\pi^{2}\int_{T}s_{k}(x_{1},x_{2})^{-4}dx_{1}dx_{2}  \approx 0.001136.$$
The Einstein constant is given by the formula (\ref{EinConst})
$$\Lambda = \sqrt{\frac{96\pi^{2}\chi(M)+144\pi^{2}\tau(M)-\int_{M}s_{k}^{2}dV_{k}}{8Vol(g_{P})}},$$
where $\chi(M)$ is the Euler characteristic of $M$ and $\tau(M)$ is the signature.  In the case of the Page metric, $\chi(\mathbb{CP}^{2}\sharp\overline{\mathbb{CP}}^{2})=4$ and $\tau(\mathbb{CP}^{2}\sharp\overline{\mathbb{CP}}^{2})=0$ yielding $\Lambda  \approx 364.44$.
\\
We now evaluate the integrals for the test function $s_{k}^{-1}$. Using (\ref{estsrecip}) we have
\begin{align*}
\|\nabla_{P}s_{k}^{-1}\|^{2}_{L^{2}(g_{P})} =& \int_{M}|\nabla_{P} s_{k}^{-1}|^{2}dV_{P}\\
 =& \frac{1}{18}\int_{M}\left(4\Lambda s_{k}^{-6}-s_{k}^{-3}\right)dV_{k} \\
 \approx & 0.0001843.
\end{align*}
The average value of $s_{k}^{-1}$, denoted $\langle s_{k}^{-1} \rangle$, is given by
$$\langle s_{k}^{-1} \rangle = \frac{1}{Vol(g_{P})}\int_{M}s_{k}^{-1} \ dV_{P}\approx 0.09559$$
and hence
\begin{align*}
\|s_{k}^{-1}-\langle s_{k}^{-1}\rangle\|^{2}_{L^{2}(g_{P})} =&\int_{M}(s^{-1}_{k}-\langle s_{k}^{-1}\rangle)^{2}dV_{e}\\
=& \int_{M}(s^{-1}_{k}-\langle s_{k}^{-1}\rangle)^{2}s_{k}^{-4}dV_{k}\\
\approx & 2.686 \times 10^{-7}.
\end{align*}
This then gives the estimate
$$ \lambda_{1}^{P} \leq \frac{ \|\nabla_{P}s^{-1}_{k}\|^{2}_{L^{2}(g_{P})}}{\|s^{-1}_{k}-\langle s^{-1}_{k}\rangle\|^{2}_{L^{2}(g_{P})}} \approx 686.2$$
and an invariant estimate
$$\lambda_{1}^{P} \leq 1.883 \Lambda.$$
\end{proof}

\begin{proof}[Proof of Theorem \ref{main} for the CLW metric]

Again we use the description that appears in \cite{HHS}. The moment polytope $P \subset \mathbb{R}^{2}$ is a pentagon which can be described as the set of points $(x_{1},x_{2}) \subset \mathbb{R}^{2}$ satisfying the inequalities $l_{i}(x)>0$ where
$$ l_{1}(x) = x_{1}, \ l_{2}(x) = x_{2}, \ l_{3}(x) = (1-x_1) , \  l_{4}(x) = (1-x_2), l_{5}(x) = (1+a-x_1-x_2).$$
Here $a$ is a constant that determines the K\"ahler class by varying the volume of the exceptional divisor when we view $\mathbb{CP}^{2}\sharp2\overline{\mathbb{CP}}^{2}$ as $(\mathbb{CP}^{1} \times \mathbb{CP}^{1})\sharp\overline{\mathbb{CP}}^{2}$.  The value of $a$ corresponding to the critical K\"ahler class has been calculated by LeBrun \cite{LeBrun} to be $a \approx 1.958$. Again, in principle, we could compute $a$ to any required accuracy as it is the solution of a polynomial equation.\\
\\
Using some of Donaldson's theory outlined in \cite{Dontor} we can calculate that the constants $c_{1}$ and $c_{2}$ that define the scalar curvature
$$s_{k}(x_1,x_2)=c_{1}(x_{1}+x_{2})+c_{2},$$
where
$$c_{1} = \frac{2}{3}(1-a^{3}) \textrm{ and } c_{2} = \frac{12(a^{5}+7a^{4}+6a^{3}+2a^{2}-5a-3)}{a^{6}+6a^{5}+9a^{4}+4a^{3}-3a^{2}-6a+1}.$$ 
As with the case of the Page metric, we give the explicit formulae for integrals of powers of the scalar curvature $s_{k}$. We first note that 
$$\int_{P}(c_{1}(x_{1}+x_{2})+c_{2})^{q}dx_{1}dx_{2} = \int_{0}^{a}(c_{1}t+c_{2})^{q}tdt+\int_{a}^{a+1}(c_{1}t+c_{2})^{q}(2a-t)dt.$$
This yields for $q \neq-1,-2$
\begin{align*}
\int_{P}(c_{1}(x_{1}+x_{2})+c_{2})^{q}dx_{1}dx_{2} =& 
\left[\frac{(c_{1}t+c_{2})^{q+1}}{c_{1}(q+1)}\left(t-\frac{(c_{1}t+c_{2})}{c_{1}(q+2)}\right) \right]_{0}^{a}\\
&+ \left[\frac{(c_{1}t+c_{2})^{q+1}}{c_{1}(q+1)}\left((2a-t)+\frac{(c_{1}t+c_{2})}{c_{1}(q+2)}\right) \right]_{a}^{1+a}.
\end{align*}
When $q=-1$ we have
\begin{align*}
\int_{P}(c_{1}(x_{1}+x_{2})+c_{2})^{-1}dx_{1}dx_{2} =& -\left[\frac{c_{2}}{c_{1}^{2}}(c_{1}t+c_{2})\log(c_{1}t+c_{2})-c_{1}t\right]_{0}^{a}\\
&+ \left[\frac{(2a-t)}{c_{1}}\log (c_{1}t+c_{2})\right]^{a}_{a+1}\\
&+\left[\frac{1}{c_{1}^{2}}(c_{1}t+c_{2})\log(c_{1}t+c_{2})-c_{1}t\right]_{a}^{a+1},
\end{align*}
and for $q=-2$
\begin{align*}
\int_{P}(c_{1}(x_{1}+x_{2})+c_{2})^{-2}dx_{1}dx_{2} =& \left[-\frac{t}{c_{1}(c_{1}t+c_{2})}\right]_{0}^{a}+\left[\frac{1}{c_{1}^{2}}\log(c_{1}t+c_{2})\right]_{0}^{a}\\
&-\left[\frac{(2a-t)}{c_{1}(c_{1}t+c_{2})}\right]_{a}^{a+1}-\left[\frac{1}{c_{1}^{2}}\log (c_{1}t+c_{2})\right]_{a}^{a+1}.
\end{align*}
The volume of the CLW metric in this representation is 
$$Vol(g_{CLW}) = \int_{M}s_{k}^{-4}dV_{k} =4\pi^{2}\int_{P}s_{k}(x_{1},x_{2})^{-4}dx_{1}dx_{2} \approx 0.5834$$
Again, the Einstein constant $\Lambda$ can be computed from the formula (\ref{EinConst}). In the case of the CLW metric, $\chi(\mathbb{CP}^{2}\sharp2\overline{\mathbb{CP}}^{2})=5$ and $\tau(\mathbb{CP}^{2}\sharp2\overline{\mathbb{CP}}^{2})=-1$ yielding $\Lambda \approx 15.09$.
\\
We now evaluate the integrals for the test function $s_{k}^{-1}$;

\begin{align*}
\|\nabla_{CLW}s_{k}^{-1}\|^{2}_{L^{2}(g_{CLW})} &= \int_{M}|\nabla_{CLW} s_{k}^{-1}|^{2}dV_{CLW} \\
&= \frac{1}{18}\int_{M}\left(\Lambda s_{k}^{-6}-s_{k}^{-3}\right)dV_{k}\\
&\approx 0.02081.
\end{align*}
We also have
$$\langle{s}_{k}^{-1}\rangle = \frac{1}{Vol(g_{CLW})}\int_{M}s_{k}^{-1} \ dV_{CLW}\approx 0.2687,$$
\begin{align*}
\|s_{k}^{-1}-\langle{s}_{k}^{-1}\rangle\|^{2}_{L^{2}(g_{CLW})} &=\int_{M}(s^{-1}_{k}-\langle{s}_{k}^{-1}\rangle)^{2}dV_{CLW}\\
&= \int_{M}(s^{-1}_{k}-\langle{s}_{k}^{-1}\rangle)^{2}s_{k}^{-4}dV_{k} \\
&\approx 0.0006545.
\end{align*}
This then gives the estimate
$$ \lambda_{1}^{CLW} \leq \frac{ \|\nabla_{CLW}s^{-1}_{k}\|^{2}_{L^{2}(g_{CLW})}}{\|s^{-1}_{k}-\langle{s}_{k}^{-1}\rangle\|^{2}_{L^{2}(g_{CLW})}} \approx 31.79$$
and an invariant estimate
$$\lambda_{1}^{CLW} \leq 2.107\Lambda.$$
\end{proof}
We state the bounds in Theorem \ref{main} to 3 significant figures. As remarked previously, greater precision in the calculation of the parameter $a$ in both cases would lead to greater precision in the bounds.  The main point is we can be confident that $\lambda^{P}_{1} <2\Lambda$. 

\section{The Matsushima theorem}\label{Disc}
The choice of $s_{k}^{-1}$ as the test function in the proof of Theorem \ref{main} might not seem the most natural. However if one takes $s_{k}$ (normalised to have integral 0) as a test function for example, then the Rayleigh quotient is
$$\frac{\|\nabla s_{k}\|^{2}_{L^{2}(g_{P})}}{\|s_{k}\|^{2}_{L^{2}(g_{P})}}\approx 1.968\Lambda$$
for the Page metric and
$$\frac{\|\nabla s_{k}\|^{2}_{L^{2}(g_{CLW})}}{\|s_{k}\|^{2}_{L^{2}(g_{CLW})}}\approx 2.231\Lambda$$
for the CLW metric. A heuristic reason for why on might expect $s_{k}^{-1}$ to give a better bound than $s_{k}$ comes from examining what happens in the K\"ahler-Einstein case.  Here one has the classical estimate due to Matsushima \cite{Mat} and later generalised by Lichnerowicz \cite{Lic}. We use the version stated in \cite{Bes}.
\begin{theorem}[Matsushima, Theorem 11.52 in \cite{Bes}]
Let $(M,g)$ be a K\"ahler-Einstein metric satisfying
$$\Ric (g)=\Lambda g \textrm{ with } \Lambda>0.$$
Then the first non-zero eigenvalue of the Laplacian on scalars $\lambda_{1}$ satisfies
$$\lambda_{1}\geq 2\Lambda.$$
Furthermore, suppose equality is achieved, then 
$$\Delta f=2\Lambda f \textrm{ if and only if } \nabla f \textrm{ is a real holomorphic vector field.}$$
\end{theorem}
So on K\"ahler-Einstein manifolds, the functions that minimise the Rayleigh quotient are the ones with holomorphic gradients.  One is led to wonder if the same might be true on the conformally K\"ahler, Einstein manifolds we are interested in. Derdzinski's Theorem \ref{Derd} says that with respect to the K\"ahler metric $g_{k}$, $\nabla_{k}s_{k}$ is a holomorphic vector field. If we consider the gradient of $s_{k}^{-1}$ with respect to the Einstein metric $g_{e}=s_{k}^{-2}g_{k}$ we see 
$$\nabla_{e}s_{k}^{-1} = -s_{k}^{-2}\nabla_{e}s_{k}=-s_{k}^{-2}(s_{k}^{2}\nabla_{k}s_{k})=-\nabla_{k}s_{k}.$$ 
Hence on the conformally Kahler, Einstein $4$-manifolds we see that $s_{k}^{-1}$ is a function that has holomorphic gradient with respect to the Einstein metric. This gives a reason why one might expect $s_{k}^{-1}$ to be a good choice of test function.  The numerical results in section \ref{NumRes} also give strong evidence that  $s_{k}^{-1}$ is close to being optimal. 
\section{Linear stability}\label{LinStab}
\subsection{The definition of the $N$ operator}
In this section we give a few more details regarding the notion of linear stability. Einstein metrics are fixed points of the Ricci flow
\begin{equation}\label{RF}
\frac{\partial g}{\partial t}=-2\Ric (g),
\end{equation}
in the sense that they evolve via homothety. Perelman \cite{Per} introduced a functional $\nu$ that is monotone increasing under the flow (\ref{RF}) except at critical points of the functional. He showed that Einstein metrics are critical points of the $\nu$-functional. Hence it is a natural question to ask whether, starting at a perturbation of an Einstein metric $g_{e}$, the flow (\ref{RF}) converges back to the Einstein metric $g_{e}$. The monotonicity property of the functional means that if the second variation of $\nu$ in the direction $h\in \Sym(TM^{\ast})$ is positive, the perturbation $h$ destabilises the Einstein metric and the flow would not converge back to $g_{e}$. The second variation formula for the $\nu$-functional was first stated by Cao, Hamilton and Ilmanen in \cite{CHI}. We recall that in this paper we follow the convention that the spectrum of the Laplacian is non-negative:
\begin{theorem}[Cao-Hamilton-Ilmanen]\label{Ndef}
Let  $(M^{n},g)$ be a closed Einstein manifold with $\Ric (g)=\Lambda g$. For $h \in \Sym(TM^{\ast})$ consider variations $g_{t}=g+th$. Then the second variation of $\nu$ energy at $g$ is 
$$\frac{d^{2}}{dt^{2}}|_{t=0} \nu(g(t)) = \frac{2}{\Lambda Vol(g)}\int_{M}\langle Nh,h\rangle dV_{g},$$
where $N$ is given by
$$N(h) = -\frac{1}{2}\nabla^{\ast}\nabla h+\Riem(h,\cdot) +\div^{\ast}\div (h) +\frac{1}{2}\nabla^{2}v_{h}-\frac{\Lambda}{nVol(g)}\int_{M}tr(h)dV_{g}g$$ 
and $v_{h}$ is the solution of
$$\Lambda v_{h}-\Delta v_{h} = \div \div(h).$$
\end{theorem}
We remark that the proof of this theorem was not given in \cite{CHI}. A more general second variation formula for the variation at a Ricci soliton was proved by Cao and Zhu in \cite{CZ}. We recall the splitting of $\Sym(TM^{\ast})$ into
$$\Sym(TM^{\ast}) = \ker(\div)_{0} \oplus \mathbb{R}g \oplus \textrm{im}(\div^{\ast}),$$
where $\ker(\div_{0})$ is the space of tensors that are divergence free and $L^{2}$-orthogonal to the metric $g$ (i.e. the integral of the trace vanishes). It is not hard to show that $N$ vanishes on $\mathbb{R}g \oplus \textrm{im}(\div^{\ast})$ and so we only consider perturbations in $\ker(\div)_{0}$. Restricted to this space one has
$$ Nh = -\frac{1}{2}\nabla^{\ast}\nabla+\Riem(h,\cdot) = -\frac{1}{2}(\Delta_{L}-2\Lambda)h$$
where 
$$\Delta_{L}h = \Delta h-2\Riem(h,\cdot)+\Ric \cdot h+ h \cdot \Ric$$
is the Lichnerowicz Laplacian. Hence an Einstein metric is linearly stable if $\Delta_{L} \geq  2\Lambda$. 
\subsection{Conformal perturbations}
A conformal perturbation is one of the form $h=ug$ for some $u\in C^{\infty}(M)$. However, it is actually convenient for us to consider a gauge equivalent perturbation. In \cite{CHI} the authors define the following operator
\begin{definition}[$S$-operator]
Let $(M,g)$ be an Einstein manifold and let $u\in C^{\infty}(M)$. Then we define $S(u) \in \Sym(TM^{\ast})$ by
$$S(u) =\left(\Lambda u-\Delta u \right)g-\nabla^{2}u.$$
\end{definition}
This operator has the following desirable property
\begin{lemma}
The tensor $S(u)$ is divergence free.
\end{lemma}
\begin{proof}
This follows from the identity $\div(\phi g) = d\phi$ and the Bochner formula
$$ \div(\nabla^{2} \phi) = -d\Delta \phi + \Ric(d\phi)$$
for a smooth function $\phi$.
\end{proof}
Without loss of generality we can assume that $S(u) \in \ker(\div)_{0}$ by adding a constant to $u$ if necessary.
\begin{theorem}[Cao Hamilton-Ilmanen]
The operator $S$ satisfies the identity
$$\Delta_L(S(u)) = S(\Delta(u)).$$
Hence any eigenfunction of $\Delta |(\ker(S))^{\perp}$ gives an eigentensor of $\Delta_{L}$ with the same eigenvalue.
\end{theorem}
\begin{proof}  Expanding out the left-hand side, we get 
$$
\Delta(\Delta(u)g) + 2 Rm((\Delta u)g, \cdot) - Ric .(\Delta u)g - (\Delta u)g.Ric 
$$
$$
- \Delta(D^2(u)) - 2 Rm(D^2(u, \cdot) - Ric .D^2u - D^2u.Ric 
$$
$$
+ \frac{(\Delta u)g}{2\tau} + 2 Rm((u)g, \cdot) - Ric .(u)g - (u)g.Ric \\
$$

Then we get cancellations as  $Rm((\Delta u)g, \cdot) = Ric .(\Delta u)g$, and similarly in the third line.

For the first term of the second line, we use the equation relating the commutator of the Laplacian and the gradient of a function: 
\begin{align*}
\Delta(D^2(u))= & D^2(\Delta(u)) + (R_{jp}g_{ik} + R_{ip}g_{jk} - 2R_{kipj})\nabla^k\nabla^p u \\
&+ (\nabla_iR_{jp} + \nabla_jR_{pi} - \nabla_pR_{ij})\nabla^p u. 
\end{align*}

Clearly the term in front of $\nabla^p u$ vanishes. The term in front of $\nabla^k\nabla^p u$ is $$-2Rm(D^2(u, \cdot) - Ric.D^2u - D^2u.Ric.$$
Hence the left-hand side of the first equation becomes 
$$
\Delta(\Delta(u)g) - D^2(\Delta(u)) + \frac{(\Delta u)g}{2\tau}
$$
and we are done. 
\end{proof}
We note that for any Einstein manifold apart form the round sphere, \mbox{$\ker(S) = \{0\}$}. Hence the Page metric is destabilised by $S(u_{1})$ where $u_{1}$ is an eigenfunction associated to $\lambda_{1}^{Page}$. 
\section{Numerical Results}\label{NumRes}
\subsection{The Page metric}
In this section we report on some work that examines numerically the spectrum of the Page metric. We begin by considering the cohomogeneity one description.  The principal orbits for the cohomogeneity one action  by $U(2)$ on $\mathbb{CP}^{2}\sharp\overline{\mathbb{CP}}^{2}$ are $\mathbb{S}^{3}$ and they form a dense subset diffeomorphic to $I\times \mathbb{S}^{3}$ for an interval $I$. Metrics for which the $U(2)$ action is isometric can be written in the form: 
$$g = dt^{2}+f^{2}(t)\sigma^{2}_{X}+h^{2}(t)(\sigma_{Y}^{2}+\sigma_{Z}^{2})$$
where $f$ and $h$ are smooth functions and  $\sigma_{X}, \sigma_{Y}, \sigma_{Z}$ are the one-forms dual to the usual generators of $\mathfrak{su}(2)$. The Einstein equation becomes a non-linear system of ODEs which one can solve explicitly (see \cite{Bes} for example). In fact we use a Runge-Kutta (RK4) integrator to numerically generate $f$ and $h$ but as the explicit formulae for $f$ and $h$ involve evaluating an integral we can get the same precision using this method. We take $\Lambda=1/2$ which corresponds to initial conditions 
$$(f(0),\dot{f}(0),h(0),\dot{h}(0)) = (0,1,2.62,0)$$
and we take a step size of $0.0001$ in the RK4 integrator. The interval $I=(0,4.6145)$ in this case.
We take the set 
$$T_{N} = \{1,t,t^{2},...,t^{N}\}$$
where $t$ is the coordinate on the interval $I =(0,4.6145)$.  We then calculate the matrices $A$ and $B$ where
$$A_{ij} = \langle\nabla t^{i-1}, \nabla t^{j-1}\rangle_{L^{2}} \textrm{ and } B_{ij}=\langle t^{i-1},t^{j-1}\rangle_{L^{2}}.$$
 Table 1 records the values of the normalised, non-negative eigenvalues of $B^{-1}A$ for various values of $N$.\\
\begin{table}[!h]
\begin{center}
\caption{Eigenvalues of $B^{-1}A$ using $T_{N}=\{1,t,...,t^{N}\}$}
\begin{tabular}{|c|c|}
\hline
\textbf{N} & \textbf{Non-negative eigenvalues of $B^{-1}A$}\\
\hline
1 & 2.0076\\
\hline
2 & 2.0076, 6.6356\\
\hline
3 & 1.8833, 6.6356, 15.178\\
\hline
4 & 1.8833, 5.5941, 15.178, 29.426 \\ 
\hline
5 & 1.8831, 5.5941, 11.269,  29.426, 51.587\\ 
\hline
\end{tabular}
\end{center}
\end{table}
\\
We also consider the Rayleigh-Ritz method using the results of Proposition \ref{Sprop}.  We take the the set
$$ T_{N} = \{1, s_{k}^{-1},...,s_{k}^{-N}\}$$
and we compute the matrices
$$A_{ij} = \langle\nabla s_{k}^{1-i}, \nabla s_{k}^{1-j}\rangle_{L^{2}} \textrm{ and } B_{ij}=\langle s_{k}^{1-i},s_{k}^{1-j}\rangle_{L^{2}}.$$
Table 2 records the normalised non-negative eigenvalues of the the matrix $B^{-1}A$ for various values of $N$.\\
\begin{table}[!h]
\begin{center}
\caption{Eigenvalues of $B^{-1}A$ using $T_{N}=\{1,s_{k}^{-1},...,s_{k}^{-N}\}$}
\begin{tabular}{|c|c|}
\hline
$\textbf{N}$ &\textbf{Non-negative eigenvalues of $B^{-1}A$} \\
\hline 
1 & $1.8830$\\
\hline
2 & $1.8830$, $5.5789$ \\
\hline
3 & $1.8830$, $5.5789$, $11.134$\\ 
\hline
4 & $1.8830$, $5.5787$, $11.131$, $24.484$\\
\hline
5 & $1.8830$, $5.5787$, $11.112$, $18.094$\\
\hline
\end{tabular}
\end{center}
\end{table}
\\
The numerical investigation seems to suggest that it would be reasonable to conclude that the $U(2)$-invariant spectrum of the Page metric begins:  
$$0, \ 1.9\Lambda, \ 5.6\Lambda, \ 11\Lambda,...$$
where the factors are taken to 2 significant figures.
\subsection{The Chen-LeBrun-Weber metric}
We again consider the Rayleigh-Ritz method using the results of Proposition \ref{Sprop}.  We take the the set
$$ T_{N} = \{1, s_{k}^{-1},...,s_{k}^{-N}\}.$$
Table 3 records the normalised non-negative eigenvalues of the the matrix $B^{-1}A$ for various values of $N$.\\
\begin{table}[!h]
\begin{center}
\caption{Eigenvalues of $B^{-1}A$ using $T_{N}=\{1,s_{k}^{-1},...,s_{k}^{-N}\}$}
\begin{tabular}{|c|c|}
\hline
$\textbf{N}$ &\textbf{Non-negative eigenvalues of $B^{-1}A$} \\
\hline 
1 & $2.1043$\\
\hline
2 & $2.0967$, $5.3423$ \\
\hline
3 & $2.0967$, $5.3363$, $8.3081$\\ 
\hline
4 & $2.0969$, $5.3746$, $10.231$\\
\hline
5 & $2.0965$, $5.3742$, $10.209$\\
\hline
\end{tabular}
\end{center}
\end{table}\\
Table 3 gives strong evidence that our bound is very close to being optimal, at least for the $\mathbb{T}^{2} \times \mathbb{Z}_{2}$ invariant spectrum of the CLW metric.  We remark that our method seems to give the next non-zero eigenvalue of the CLW metric as close to $5.37\Lambda$. It would be intriguing to use the Bunch-Donaldson approximation to the CLW metric to numerically investigate the spectrum and see if the bound $2.11\Lambda$ is also close to optimal. One could also investigate whether there are other eigenvalues apart from those in the  $\mathbb{T}^{2} \times \mathbb{Z}_{2}$-invariant spectrum. We leave this as a project for the future.

\end{document}